\documentclass[reqno,11pt]{amsart}
\usepackage{latexsym,amssymb,amsfonts,amsmath,amsthm}
\usepackage{mathrsfs}
\usepackage[usenames]{color}
\usepackage{eucal}
\usepackage{enumitem}
\usepackage[colorlinks,citecolor=blue,urlcolor=blue]{hyperref}

\newtheorem{thm}{Theorem}[section]
\newtheorem{cor}[thm]{Corollary}
\newtheorem{lem}[thm]{Lemma}

\theoremstyle{definition}

\theoremstyle{remark}
\newtheorem{rem}[thm]{Remark}
\numberwithin{equation}{section}
\newcommand{\Real}{\mathbb R}
\newcommand{\eps}{\varepsilon}

\newcommand{\F}{\mathcal{F}}

\renewcommand{\P}{\mathbb{P}}

\newcommand{\E}{\mathbb{E}}
\newcommand{\N}{\mathbb{N}}

\newcommand{\Var}{\mathrm{Var}}

\begin{document}

\title[Population dynamics: from small and random to large and pseudo-deterministic]{Populations with interaction and environmental dependence: from few, (almost) independent, members into deterministic evolution of high densities.}

\author{P. Chigansky}
\address{Department of Statistics,
The Hebrew University,
Mount Scopus, Jerusalem 91905,
Israel}
\email{pchiga@mscc.huji.ac.il}

\author{P. Jagers}
\address{Mathematical Sciences, Chalmers University of Technology
and University of Gothenburg, SE-412 96 Gothenburg, Sweden}
\email{jagers@chalmers.se}

\author{F. Klebaner}
\address{School of Mathematical Sciences,
Monash University, Monash, VIC 3800, Australia}
\email{fima.klebaner@monash.edu}


\date{}
\begin{abstract}
Many populations, e.g. of cells, bacteria, viruses, or replicating DNA
molecules, but also of species invading a habitat, or physical systems
of elements generating new elements, start small, from a few lndividuals, 
and grow large into a noticeable fraction of the environmental
carrying capacity $K$ or some corresponding regulating or system scale
unit.  Typically, the elements of the initiating, sparse set will not be  
hampering each other  and their number will grow from $Z_0=z_0$  in a branching process or Malthusian like, roughly exponential
fashion, $Z_t \sim a^tW$, where $Z_t$ is the size at discrete time $t\to\infty$, 
$a>1$ is the offspring mean per individual (at the low starting
density of elements, and large $K$), and $W$
a sum of $z_0$ i.i.d. random variables. It will, thus, become
detectable (i.e. of the same order as $K$) only after around  $\log K$
generations, when its density $X_t:=Z_t/K$ will tend to be strictly
positive. Typically,  this entity will be random, even if the very beginning
was not at all stochastic, as indicated by lower case $z_0$, due to 
variations during the early development. However, from that time onwards,
law of large numbers effects will render the process deterministic, though 
inititiated by the random density at time log $K$, expressed through the 
variable $W$. Thus, $W$ acts both as a random veil concealing the start and 
a stochastic initial value for later, deterministic population density development.   

We make such arguments precise, studying general
density and also system-size dependent, processes, as
$K\to\infty$. As an intrinsic size parameter,  $K$  may also
be chosen to be the time unit. The fundamental ideas are to 
couple the initial system
to a branching process and to show that late densities develop very much like 
iterates of a conditional expectation operator.

The ``random veil'', hiding the start,  was first observed in the very 
concrete special case of finding the initial copy number in
quantitative PCR under Michaelis-Menten enzyme kinetics, where the
initial individual replication variance is nil if and only if the
efficiency is one, i.e. all molecules replicate
\cite{JagKlePCR}. 

\end{abstract}

\maketitle

\section{Introduction: Replication with interaction and dependence}

We consider sets of elements where, in principle, each element may
generate new elements. For lucidity we regard time as discrete,
labeling it by $n=0,1,2,\ldots$, referring to it also as generations,
cycles or rounds, and call the system a 'population' of 'individuals',
even though we have quite general such structures in mind. This is
like in branching processes but without independence between
individuals required. 
Here, we assume that the individual
offspring generation (reproduction, replication or whatever) may be
influenced by a system 'carrying capacity', $K$, which we think of as
large, as compared to the population starting number $Z_0$, and also
by the number of other individuals present. We say that replication is
capacity and population size, or 'density' dependent. 

The process definition is patterned after the recursive scheme used to
build up Galton-Watson processes: Let
$$ \xi_{n,j}, n\in\N,j\in\N,$$
be non-negative integer-valued random variables, where we think of
$\xi_{n,1}, \xi_{n,2}, \ldots$ as the possible offspring numbers of
various individuals in the $n-1$:th generation. Thus, we define 
$\{Z_n, n=0,1,2,\ldots \}$, by the initial number  $Z_0$ and  
\begin{equation}
\label{Zn}
Z_n =   \sum_{j=1}^{Z_{n-1}}  \xi_{n,j}.
\end{equation}
The dependence structure is made precise in a basic assumption:

Assumption A0. For each fixed $n\in \N$, the $\xi_{n,j}, j=1,2, \dots$,  
are {\it conditionally} independently and identically
distributed, given the preceding, $\F_{n-1} := \sigma(\{\xi_{k,j}, k<n,
j\in\N\})$. The process is Markovian in the sense that the conditional
distribution should be determined by the couple $K$ and
$X_{n-1}=Z_{n-1}/K$, the {\em carrying capacity} and {\em population density}, in
such a manner that the variables increase in distribution with $K$ and
decrease with $x=X_{n-1}$, the limiting distribution, as $K\to\infty$, 
{\it the asymptotic reproduction}, 
being proper for each $x\in\Real _+$. 

Three entities pertaining to the density  turn out to be crucial
for the analysis of process start and late development. 
They are:
\begin{enumerate}
\item the conditional mean number of offspring per individual,
$$m^K(X_{n-1}) = \E[\xi_{n,i}|\F_{n-1}],$$ 
\item the corresponding variance, 
$$\sigma_K^2(X_{n-1}) = {\rm Var}[\xi_{n,i}|\F_{n-1}],$$
\item and the conditional expectation of the density process,
$$f^K(x) =\E[X_n|X_{n-1}=x]=xm^K(x),$$ 
\end{enumerate}
where the dependence of variance and expectation operators upon $K$ is implicit.
From A0 it follows 
that the $m^K$ form a non-decreasing sequence of non-increasing
functions, and hence must have a non-increasing limit, $m$. The means
and variances $m^K$ and $\sigma_K^2$ are supposed defined on all of $\Real _+$. 

We formulate boundedness and smoothness criteria for these functions,
which will lead to classical Malthusian growth for $Z_n$ in an early
stage, $n \leq n_K=c\log K, 0<c<1$. Then, we make use of a law of
large numbers for branching processes with a
  threshold and density dependence,  \cite{Klethr}, and a
  large initial value, in our case  $O(K^c)$,  at round $n_K$. All 
logarithms are with base $a$.

The assumptions beyond A0 are:

\begin{enumerate}[label={A\arabic*}]

\item\label{A1}  The limiting expected conditional reproduction given a population density $x$, $m(x)$ has a derivative which is uniformly continuous in a neighbourhood to the right of the origin, and $a=m(0)>1$. 
As $K\to\infty$, the continuous non-increasing functions
  $m^K$ converge uniformly to a bounded differentiable function $m$,
  $0\leq m(x)-m^K(x)\leq Cx + o(x)$ for some $C>0$ and uniformly in $K$, as $x\to 0$. 

\item\label{A2}   The limiting conditional expected density $f, f(x)=xm(x), x\geq 0$, is strictly increasing. 

\item  As $K\to\infty$, $X_0$ converges in probability to some limit $x_0\geq 0$. In particular, if there is a fixed starting number,  $x_0= 0$.

\item\label{A4}  The conditional reproduction variance $\sigma_K^2(x)=\mathrm{Var}[\xi_{n,i}|X_{n-1} =x]$ is uniformly bounded
and, as $K\to\infty$, converges to some $\sigma^2(x)$ uniformly. The latter, hence, is also bounded.

\item\label{A5}  There is a constant $C>0$ such that uniformly for all $K$ 
$$
a\geq m^K(x) = m^K(0)-Cx +o(x)\quad \text{as}\ x\to 0.
$$
 Further, $0\leq a-m^K(0)= O(1/\sqrt K)$ and also $\sup_{x\ge 0}|f^K(x)-f(x)| = O(1/\sqrt K)$. 
\end{enumerate}

These look like innocuous smoothness requirements, but \ref{A2} contains something more. For fixed $K$ dependence upon the density is the same as dependence on population size, and it thus seems little of a restriction to ask that the next generation should tend to increase with density. It might however be argued that there could be a density above which for example no replication is possible. This in a sense, however, would introduce a sort of further carrying capacity, besides $K$.


\section{In spite of interaction and capacity dependence,  the
  beginning looks like branching, when the carrying capacity becomes large.} 
 
An approximating process, at low density and high carrying capacity, 
$\tilde Z=\{\tilde{Z}_n\}$, is crucial in the analysis. It has the same starting
number $Z_0=z_0$ as the original process, but then it continues as a
classical Galton-Watson process,
\begin{eqnarray}\label{tilden}
\tilde Z_0 & = &z_0\\
\tilde Z_{n} & = & \sum_{j=1}^{\tilde Z_{n-1}}\eta_{n,j},
\end{eqnarray}
where the $\eta$ variables are i.i.d. with the asymptotic reproduction distribution
as $K\to\infty$ for $x=0$. Thus, $\E[\eta]=a=m(0)>1$ and
$\Var[\eta]=\sigma^2(0)<\infty$.(Lower case $z_0$ indicates an unknown but deterministic starting number.)

From classical branching process theory,  $W(z_0):=\lim_{n\to\infty}
\tilde{Z}_n/a^n$ will have the distribution of $z_0$ independent
$W$-copies, each with expectation 1 and variance
$\sigma^2(0)/(a^2-a)$. In particular, 
$$ \tilde Z_{n_K}/K^c = \tilde Z_{n_K}/a^{n_K} \sim W(z_0).$$
 By approximation, this extends to:
\begin{thm} Under assumptions A0, \ref{A1}, and \ref{A5},  $Z_{n_K}/a^{n_K} \to  W(z_0)$ and
thus $X_{n_K}\sim  W(z_0)K^{c-1}$,   in probability and $L^1$, as $K\to\infty$.
\end{thm}
\begin{proof}
Construct the replication processes $Z$ and $\tilde Z$, as well as a
third process $Z^\gamma=\{Z^\gamma_n\}$, on the same probability space
by the following coupling. Let $U_{n,j}$, $n,j\in \mathbb{N}$
be independent uniformly distributed random variables on
$[0,1]$. For each $K$ and $x$ define $t^K_{-1}(x)=t_{-1}=0$ and $0\leq t^K_0(x)\leq t^K_1(x)\leq t^K_2(x)\leq\ldots$
so that $\P(t^K_{k-1}(x)< U_{n,j}\leq
t^K_{k}(x))=\P(\xi_{n,j}=k|X_{n-1}=x), k\in \mathbb{N}$. Further, let $0\leq t_0\leq t_1\leq
t_2\leq\ldots$ so that $\P(t_{k-1}< U_{n,j}\leq
t_{k})=\P(\eta_{n,j}=k), k\in \mathbb{N}$. We can then define the
reproduction random variables $\xi_{n,j}$ and population sizes  $Z_n$, $\tilde Z_n$, as well as densities $X_n$
inductively on the same probability space by, 
$$\xi_{n,j} = k \Leftrightarrow t^K_{k-1}(X_{n-1})< U_{n,j}\leq
t^K_{k}(X_{n-1})\mbox{ and }\eta_{n,j} = k \Leftrightarrow t_{k-1}< U_{n,j}\leq t_{k}.$$ and, as before,  \eqref{Zn} and \eqref{tilden}, using $\xi_{n,j}$ and $\eta_{n,j}$ respectively.
 Similarly, we write 
$$\xi^\gamma_{n,j}:=\sum_{k=0}^\infty k1_{(t^K_{k-1}(K^{\gamma-1}),t^K_{k}(K^{\gamma-1})]}(U_{n,j})$$ and $Z^\gamma_n$ correspondingly.

By the distributional properties of $\xi_{n,j}|X_n=x$ and $\eta_{n,j}$ (Assumption A0),
$$
t^K_k(x) = \P(\xi_{n,j}\leq k|X_n=x) \geq \P(\eta_{n,j}\leq k)=t_k
$$
and $$t^K_k(K^{\gamma -1})\geq t^K_k(X_n),$$
the latter as long as $n<\tau:=\inf\{n;X_n>K^{\gamma-1}\}$. 
Hence, by induction
for the random entities realised with the help of the $U_{n,j}, \; \tilde Z_n\geq Z_n, n\in\mathbb N,$ pointwise, and $Z^\gamma_n\leq Z_n, n<\tau$. Further, $\tilde Z_n\geq Z^\gamma_n$ for all $n$. It follows that 
$$0\leq\tilde Z_n-Z_n \leq \tilde Z_n - Z^{\gamma}_n 1_{\{n<\tau\}} -
Z_n1_{\{n\geq\tau\}}  \leq \tilde Z_n - Z^{\gamma}_n 1_{\{n<\tau\}}\leq \tilde
Z_n - Z^{\gamma}_n  + Z^{\gamma}_n1_{\{n\geq\tau\}}.$$

Now, in order to show  
that
$$
\lim_{K\to\infty}(\tilde Z_{n_K}-Z_{n_K})K^{-c} =0,
$$
we choose  $1/2 < c<\gamma <1$.
By this and \ref{A5}, 
\begin{align*}
a\geq m^K(K^{\gamma-1})=\, & a - \big(a-m^K(0)\big) - \big(m^K(0)- m^K(K^{\gamma-1})\big)\geq \\
& a - aAK^{-1/2} -aBK^{\gamma -1}+ o(K^{\gamma -1})
= a\big(1 - BK^{\gamma -1}+ o(K^{\gamma -1})\big)
\end{align*}
for suitable constants $A$ and $B$. 
Hence,
\begin{align*}
\E\big(\tilde Z_{n_K} - Z^\gamma_{n_K}\big) =\, & 
z_0 \big(a^{n_K} - m^K(K^{\gamma  -1})^{n_K}\big) = \\
&
z_0a^{c\log K}\big(1 - (1-BK^{\gamma-1} +o(K^{\gamma-1}))^{c\log K}\big) = o(K^{c}). 
\end{align*}
Thus,
$$\E[\tilde X_{n_K} - X^\gamma_{n_K}]=o(K^{c-1}).$$

For the remaining term, note that
$$\E[Z^\gamma_{n_K};\tau\leq n_K]\le \E [\tilde Z_{n_K};\tau\le n_K]
\le \Big(\E[ {\tilde Z}^2_{n_K}]\P(\tau \le n_K)\Big)^{1/2},$$
by the Cauchy-Schwartz inequality.
Since $Z_n\le \tilde Z_n$ for all $n$, it takes longer for the former
process to reach $K^\gamma$  than for the latter, so that
$$\tau\ge \nu :=\inf\{n:\tilde Z_n>K^\gamma\}$$
and
\begin{align*}
\P (\tau\le n_K)\le\, & \P \big(\nu \le n_K\big)
=  \P \big(\sup_{n \le n_K}\tilde Z_n>K^\gamma\big)= \\
&
\P \big(a^{-n_K}\sup_{n\le n_K}\tilde Z_n>K^\gamma a^{-n_K}\big)\le \\
& \P \big(\sup_{n \le n_K}\tilde Z_na^{-n}>K^{\gamma -c}\big)\le K^{c-\gamma},
\end{align*}
where the last bound is Doob's inequality for the martingale $\{\tilde Z_na^{-n}\}$. Since
$\E[\tilde Z^2_{n_K}] =  O(K^{2c})$, by the formula for expectation and
variance of Galton-Watson processes, 
$$
\lim_{K\to\infty} K^{-c}\E [\tilde Z_{n_K}; \nu\le  n_K]=0.
$$
Recalling  that $\gamma > c$, we conclude that 
\begin{equation}
\lim_{K\to\infty}(\tilde Z_{n_K}-Z_{n_K})K^{-c} =0.
\end{equation}
holds in $L^1$, and hence in probability.
For the corresponding densities, division by $K$ yields
$$
\lim_{K\to\infty}(\tilde X_{n_K}-X_{n_K})K^{1-c} =0.
$$

 \end{proof}

\section{The branching like stage forms a random initial condition for later development.}

If the process does not die out, it will thus grow
exponentially in n, at least as long as it does not approach $K$ and
for fixed $K$ this holds for some $n_K=c\log K$ generations. 
Then,  law of large numbers type effects should stabilise the subsequent
growth.We proceed to this, giving first a result on densities for
fixed time, $K\to\infty$, and $K$-dependent but stabilising starting density:

\begin{thm}\label{Fima} Under the assumptions stated,
for any time $n$,
$$\lim_{K\to\infty}X_n= x_n,$$ in probability, where $x_n$ is the
$n$:th iterate of $f$, $x_{n}=f_n(x_0)$. If $X_0\to x_0$ holds in $L^1$,
the conclusion can actually be strengthened to mean square
convergence.
\end{thm}
For a proof (under somewhat weaker conditions), see Theorem 1 of \cite{Klethr}.

Now, in our framework the second, ``post-branching'', stage starts
at time $n_K$ from $W(z_0)a^{n_K} =W(z_0)K^c$ individuals. Hence, the
starting density, $W(z_0)K^{c-1}\to x_0=0$. But this is a fixed point
of $f$, and so Theorem \ref{Fima} just yields convergence to zero. The
remedy is to consider ever later time points. 

\begin{lem} If $f$ increases strictly but $m$ decreases and $a=m(0)>1$ 
(Assumptions A0, \ref{A1}, and \ref{A2}),then
$$
h(x)=\lim_{n\to\infty}f_n(x/a^n).
$$
is well defined, continuous, and strictly increasing.The convergence
is uniform and $h(0)=0$.
\end{lem}

\begin{proof}
Since $f$ is increasing,  so are all $f_n$. By definition,
 $$f(x/a)=m(x/a)x/a\leq m(0)x/a=x,$$ 
for $x \geq 0$. 
Hence,
 $$f_{n+1}(x/a^{n+1})=f_n(f(x/a^{n+1}))\leq f_n( x/a^{n}).$$
The sequence $h_n(x) := f_n(x/a^n)$ thus decreases in $n$ for any positive
$x$, and its limit  $h$, as $n\to\infty$, must exist and be a non-decreasing function, like the $f_n$. By Dini's theorem, the
convergence is uniform on any compact interval. Clearly, $h(0)
=h_n(0)=f_n(0)=0$ for all $n$.



It remains to prove that the limit $h$ increases strictly. However, there exist $C>0$ and $ \epsilon >0$ such that
$$f'(x)=m(x)+xm'(x)=m(0)+m(x)-m(0)+xm'(x)\geq a-2x\sup_{0\leq u\leq x}|m'(u)|>a-Cx>0,$$
for $0<x<\epsilon$. 
For any $x<\min(\epsilon,1/C)$,
\begin{align*}
h_n'(x)  =& a^{-n} f'_n(x/a^n) = a^{-n}\prod_{j=0}^{n-1} f'(f_j(x/a^n)) \ge 
a^{-n}\prod_{j=0}^{n-1} \big(a- Cf_j(x/a^n)\big) \ge\\
& 
a^{-n}\prod_{j=0}^{n-1} \big(a- Cx a^{j-n}\big)
\ge \prod_{j=0}^{n-1} \big(1-  a^{-j}\big)\ge e^{-a}, \quad \forall \, n\ge 0.
\end{align*}
Taking the limit $n\to\infty$, we see that $h$ increases strictly  in an open neighbourhood of the origin. However, as  $f_{n+1}(x/a^{n+1})=f(f_n((x/a)/a^{n}))$,  letting  $n\to\infty$, shows that $h$ solves Schr\"{o}der's functional equation
$$
h(x)=f(h(x/a)). 
$$
Therefore, if it were constant on an interval $[x_1,x_2]$ with $x_2>x_1$, then also $h(x_1/a^k) = h(x_2/a^k)$, for any integer $k\ge 1$, contradicting the fact that $h$ increases strictly on some neighbourhood of the origin. Thus,
$h$ must be strictly increasing on the positive half line.

\end{proof}

An immediate consequence of this will be of explicit use in the proof
of our main Theorem \ref{main} later:
\begin{cor} 
\begin{equation}
\lim_{n\to\infty}f_n(x/a^n +o(a^{-n})) = h(x).
\end{equation}
\end{cor}

Now, for fixed $K$, the density process $X$ satisfies the fundamental
recursive equation (cf. \cite{Klethr})
\begin{equation}
\label{barZ}
X_n = f^K(X_{n-1})  + \frac 1{\sqrt K}\eps_{n},
\end{equation}
where
$$
\eps_n =  \frac 1{\sqrt K}\sum_{j=1}^{KX_{n-1}} (\xi_{n,j}-\E[\xi_{n,j}|\F_{n-1}])
$$
a martingale difference sequence,
 $\E[\eps_n|\F_{n-1}]=0$, with
$$
\E [\eps^2_n|\F_{n-1}] = \mathrm{Var}[\eps_n|\F_{n-1}]= \sigma_K^2(X_{n-1})  .
$$

The corresponding deterministic recursion, obtained by omitting the
martingale difference term, is
$$
x^K_n=f^K(x^K_{n-1})=f^K_n(x_0).
$$

From now on, what is needed of Assumptions A0-\ref{A5} is used freely. We take $1/2<c<1$, write $\nu_K=\log K - n_K= (1-c)\log K$, 
and interpret $X$ suffices as their integral parts.
\begin{lem}
$$
X_{\log K}-f^K_{\nu_K}(X_{n_K})\xrightarrow[K\to\infty]{L^1} 0.
$$
\end{lem}
\begin{proof}
Write 
$$\Delta_n = X_{n}-f^K_{n-n_K}(X_{n_K}),$$
for $n>n_K$.Then,
$$\Delta_n = f^K(X_{n-1})+\frac 1{\sqrt K}\eps_{n}-f^K\circ f^K_{n-1-n_K}(X_{n_K}).$$ 
Since for any $x\geq 0$, $0\leq \frac d{dx}f^{K}(x)=m^K(x)+x\frac d{dx} m^{K}(x)\leq
m^K(x)\leq a$ by assumption,
\begin{align*}
|\Delta_n|\leq & 
|f^K(X_{n-1})-f^K\circ f^K_{n-1-n_K}(X_{n_K})|+|\frac 1{\sqrt K}\eps_{n}|\leq \\
& a|\Delta_{n-1}|+|\frac 1{\sqrt
  K}\eps_{n}|\leq\ldots\leq \sum_{j=0}^{n-n_K-1}a^j|\frac 1{\sqrt
  K}\eps_{n-j}|. 
\end{align*}  
Adding to this that, for any natural $k$,
$$\E[|\eps_k|\leq \sqrt{\E[\eps_k^2]}=
\sqrt{\E[\E[\eps_k^2|X_{k-1}]]}\leq\sup_x\sigma_K(x) < \infty,$$ 
by \ref{A4}, we can conclude that
$$\E[|\Delta_{\log K}|]\leq Ca^{\nu_K}\frac 1{\sqrt
  K}\sup_x\sigma_K(x)=CK^{1/2-c}\sup_x\sigma_K(x)\to 0, $$
as $K\to\infty$.
\end{proof}

\begin{lem}
$$f^K_{\nu_K}(X_{n_K})-f_{\nu_K}(X_{n_K})\to 0.$$
\end{lem}
\begin{proof}
By Assumption \ref{A5}, for any $0\le x\le 1 $ and some $C>0$, 
\begin{align*}
|f^K_n(x)-f_n(x)| \leq & |f^K\circ f^K_{n-1}(x) - f\circ f^K_{n-1}(x)|
+ |f\circ f^K_{n-1}(x)-f\circ f_{n-1}(x)| \leq \\
& C/\sqrt{K} +
a|f^K_{n-1}(x) - f_{n-1}(x)|.
\end{align*}

Hence, by induction, for any $x$ and $n$
$$|f^K_n(x)-f_n(x)| \leq \frac{C}{\sqrt K}\sum_{j=0}^{n-1}a^j \leq
\frac{C}{(a-1)\sqrt K}a^n,$$ 
and 
$$\sup_{0\le x}|f^K_{\nu_K}(x)-f_{\nu_K}(x)| =
 O( a^{\nu_K}/\sqrt K) = O(K^{1/2-c})\to 0,$$
as $K\to\infty$.
\end{proof}

After these lemmas and the corollary, the proof of the main theorm below is direct.
\begin{thm}\label{main}
Assume $Z_0=z_0$ given and all of Assumptions A0-\ref{A5} valid.   Then
$X_{\log K}$ converges in distribution   
$$
X_{\log K} \xrightarrow[K\to\infty]{ D} h\circ W(z_0).
$$
\end{thm}

\begin{rem}
The limits increase strictly with $n$. 
Recall that logarithms are with
base $a$.
\end{rem}

\begin{cor} 
For any fixed $n$
$$
X_{\log K+n} \xrightarrow[K\to\infty]{ D} f_n\circ h\circ W(z_0),
$$
where   $f_n$ still denotes the $n$-th iterate of $f$.
This  extends to weak convergence of the sequences, regarded as random elements in 
$\mathbb{R}^{\mathbb{Z}}$:
$$
\{X_{\log K+n}\}_{-\infty}^{\infty} \xrightarrow[K\to\infty]{ D}
  \{f_n\circ h\circ W((z_0)\}_{-\infty}^{\infty}.
$$
\end{cor}

\begin{proof}
This follows by induction on $n$ from the fundamental representation
\eqref{barZ}. For $n=0$ it is the statement of the main result. For
$n=1$ take limits as $K\to \infty$ in \eqref{barZ}, and note that the
stochastic term vanishes. Similarly, if it holds true for $n$, it
follows for $n+1$. Functional convergence follows from 
finite dimensional convergence, {\em cf.} \cite{B}, p. 19.
\end{proof}

\begin{cor}
For any sequence $\lambda_K=o(\log K)$, $$X_{\lambda_K}
\xrightarrow[K\to\infty]{ L^1} 0.$$
\end{cor}
\begin{proof}
This is direct from
$$\E[X_{\lambda_K}]\le a^{\lambda_K-\log K}\to 0, \mbox{ as }
K\to\infty.$$
\end{proof}

This means that there is a very particulular scale, $O(\log K)$, at
which an interesting weak limit is obtained, whereas slower or faster
rates result in simpler convergences, as exhibited.

\section{Concluding remarks}

Measuring population or set size in density rather than numbers,
i.e. in capacity units, invites making the corresponding time change
into an intrinsic scale also with unit $K$,
$\bar X_t:= X_{tK} = Z_{[tK]}/K, t\ge 0$. For that process Theorem
\ref{main} yields that
 $$\bar X_0\leftarrow \bar X_{(\log K)/K} = X_{\log K}
\xrightarrow[K\to\infty]{ D} h\circ W(z_0).$$
Thus, the process in the intrinsic time scale seems to have started
from a random number of first elements, unless the variance of $W$ is
zero. Only in that case, corresponding to a completely deterministic
initial reproduction or replication process, can the the number $z_0$
of ancestors or corresponding originators be recovered behind the
random veil of history, by inversion of $h$ \cite{ChJagKle}. 

As mentioned, this article was sparked by the concrete problem of
finding the number of original templates in PCR and answering
questions about single- or multicell origin of cancers. It is,
however,  tempting to suspect that similar patterns of late observed
growth with unknown, seemingly random, origin may occur in many other
contexts.

\subsection*{Acknowledgement}
Research supported by ARC Grant DP150103588.  P. Jagers was further supported by the Knut och
Alice Wallenberg Foundation and F. C. Klebaner by ARC through DP150102758.


\begin{thebibliography}{10}


\bibitem{BHKK}  Barbour, A. D.; Hamza, K.; Kaspi, Haya; Klebaner, F. C. Escape from the boundary in Markov population processes. Adv. in Appl. Probab. 47 (2015), no. 4, 1190-1211.

\bibitem{Barbour}  A. D. Barbour, P. Chigansky, and F. C. Klebaner, On
  the emergence of random initial conditions in fluid limits. {\em J. Appl. Probab}, 2016.

\bibitem{B} Billingsley, Patrick. Convergence of probability measures. John Wiley \& Sons, 2013.
\bibitem{ChJagKle} P. Chigansky, P. Jagers, and F. C. Klebaner,
  \newblock What can be observed in real time PCR and when does it
  show?  {\em J. Math. Biol.} 76:3, 679-695  (2018)  

\bibitem{Haccou} P.\ Haccou, P.\ Jagers, and V. A.\ Vatutin, {\em Branching Processes: Variation, Growth, and Extinction of Populations}, p. 155. Cambridge University Press, Cambridge, 2005.
\bibitem{pj} P. Jagers, {\em Branching Processes with Biological Applications}, p. 33. John Wiley \& Sons, London, 1975
\bibitem{JagKlePCR} P.\ Jagers and F.\ Klebaner.
\newblock Random variation and concentration effects in {PCR}.
\newblock {\em J. Theoret. Biol.}, 224(3):299--304, 2003.
 
\bibitem{Klethr}
F.~C.\ Klebaner.
\newblock Population-dependent branching processes with a threshold.
\newblock {\em Stochastic Process. Appl.}, 46(1):115--127, 1993.

\bibitem{barebones}
F.~C.\ Klebaner, S.\ Sagitov, V.~A.\ Vatutin, P.\ Haccou and P.\ Jagers.
\newblock Stochasticity in the adaptive dynamics of evolution: the bare bones.
\newblock {\em J. Biol. Dyn.}, 5(2):147--162, 2011.

%
\bibitem{Kurtz}
Thomas~G. Kurtz.
\newblock Solutions of ordinary differential equations as limits of pure jump Markov processes. {\em J. Appl. Prob.}, 7:49--58, 1970.
 


 
  \end{thebibliography}
\end{document}